\documentclass[11pt]{amsart}
\usepackage{amsfonts,amssymb,amsmath,amsthm}

\title{Weighted Discriminants and Mass Formulas for Number Fields}
\author{Silas Johnson}

\theoremstyle{plain}
\newtheorem{theorem}{Theorem}
\newtheorem{lemma}[theorem]{Lemma}
\newtheorem{cor}[theorem]{Corollary}
\newtheorem{conj}[theorem]{Conjecture}
\newtheorem{prop}[theorem]{Proposition}

\newtheorem{question}[theorem]{Question}

\theoremstyle{definition}
\newtheorem{defn}[theorem]{Definition}

\newtheorem*{ex}{Example}

\newtheorem*{remark}{Remark}

\newcommand{\C}{\mathbb{C}}
\newcommand{\R}{\mathbb{R}}
\newcommand{\Q}{\mathbb{Q}}
\newcommand{\Z}{\mathbb{Z}}

\newcommand{\gal}{\operatorname{Gal}}
\newcommand{\disc}{\operatorname{Disc}}
\newcommand{\diff}{\mathcal{D}}
\newcommand{\pa}{\mathfrak{p}} 
\newcommand{\cw}{\ensuremath{c_{w}}}
\newcommand{\ow}{\bar{w}}
\newcommand{\cent}{C}

\begin{document}
	
\maketitle

\begin{abstract}
We define the notion of a weighted discriminant and corresponding counting function for number fields, and what it means for these counting functions to have a mass formula for a set of primes.  We extend a result of Kedlaya to show that any proper counting function for a finite group $\Gamma$ has a mass formula for the set of primes not dividing $|\Gamma|$.  We also prove that if $\Gamma$ is an $\ell$-group for some prime $\ell$, then there are only finitely many weighted discriminant counting functions for $\Gamma$-extensions of $\Q$ that have a mass formula for all primes.  Finally, we enumerate all such counting functions for $\Gamma=D_4$ and $\Gamma=Q_8$.
\end{abstract}

\section{Introduction}
\label{sec:intro}

A standard question in arithmetic statistics asks:
\begin{question}
	\label{q:field-counting}
	Let $n \ge 2$ be an integer, and $X > 0$.  Given a finite group $\Gamma$, how many number fields $K$ are there with $[K:\Q] = n$, $\gal(K/\Q) \simeq \Gamma$, and $\disc(K/\Q) < X$?  What is the asymptotic behavior of this quantity as $X \to \infty$?
\end{question}

A natural heuristic assumption to use in counting such fields is that the local completions of such fields at different primes should behave independently of each other.  Furthermore, each possible \'{e}tale extension $K_p/\Q_p$ should occur as the local completion at $p$ with frequency inversely proportion to $|\operatorname{Aut}(K_p/\Q_p)|$.  (Or, alternatively, letting $G_{\Q_p}$ be the absolute Galois group of $\Q_p$ and considering the continuous homomorphism $\phi_p:G_{\Q_p} \to \Gamma$ corresponding to $K_p/\Q_p$, each such $\phi_p$ should occur equally often.)

If this assumption holds, then studying the finite set $S_{\Q_p,\Gamma}$ of maps $\phi_p:G_{\Q_p} \to \Gamma$ can yield deep insights into the asymptotics in Question \ref{q:field-counting}.  We thus seek \textit{mass formulas} that relate the sets $S_{\Q_p,\Gamma}$ to each other for different primes $p$.  If $\Gamma$ has a mass formula that holds for every $p$, then we call this mass formula \textit{universal}; we will define these terms rigorously in Section \ref{sec:CFMF}.

One could also ask Question \ref{q:field-counting}, but with the discriminant replaced by some other invariant.  For example, we could count fields by the Artin conductor of some representation of $\Gamma$ (notably, as in \cite{MMW-MF} and \cite{ASVW-D4} for $\Gamma=D_4$).  The notion of a \textit{counting function} formalizes a set of invariants that are reasonable to substitute for the discriminant.

In this paper, we study a particular type of counting function called a \textit{natural weighted discriminant counting function}, defined in Section \ref{sec:wdisc}.  For a quartic $D_4$ field $L/\Q$ with a quadratic subfield $K$, the discriminant and conductor (of the 2-dimensional irreducible representation of $D_4$) can both be expressed in terms of $\disc(L/K)$ and $\disc(K/\Q)$.  Our weighted discriminant counting functions generalize this idea, building invariants out of discriminants of subextensions.

In particular, we are interested in studying, for a fixed $\Gamma$, which such counting functions have universal mass formulas.  For example, for quartic $D_4$ fields, the discriminant does not have a universal mass formula, but the conductor does, as first studied by Wood in \cite{MMW-MF}.

We will prove later that of all natural weighted discriminant counting functions for $D_4$, only the one corresponding to the conductor has a universal mass formula.  In general, universal mass formulas are rare.  In fact, our main theorem, proved in Section \ref{sec:big-proof}, is:

\begin{theorem}
	\label{thm:finiteMF}
	Let $\Gamma$ be any finite $\ell$-group. Then there are only finitely many positive weighted discriminant counting functions $\Gamma$ which have a universal mass formula.
\end{theorem}

On the other hand, we show in Theorem \ref{thm:mf-coeffs} (which is a slight generalization of a result of Kedlaya \cite{Kedlaya-CF}) that any sufficiently ``nice'' counting function must have a \textit{tame mass formula}, which is almost universal except for possible bad behavior at primes dividing $|\Gamma|$.

Our method of proof for Theorem \ref{thm:finiteMF} will be to understand the behavior of this tame mass formula at $\ell$, then show that in order for it to be universal, certain invariants of the counting function (the \textit{overall weights}, defined in Section \ref{sec:overall-wt}) must be bounded, which will leave only finitely many counting functions that can be universal.

\section{Counting Functions and Mass Formulas}
\label{sec:CFMF}

Fix a finite group $\Gamma$.

Let $S_{\Q_p,\Gamma}$ be the set of continuous homomorphisms $\phi : G_{\Q_p} \to \Gamma$, where $G_{\Q_p}$ denotes the absolute Galois group of $\Q_p$.
We define a \textit{counting function} for $\Gamma$ to be any map
\[
 c : \bigcup_{p\text{ prime}} S_{\Q_p,\Gamma} \to \R
\]
satisfying the following conditions:
\begin{itemize}
\item $c(\phi) = c(\gamma \phi \gamma^{-1})$ for any $\gamma \in \Gamma$ and any $\phi \in S_{\Q_p,\Gamma}$, and
\item $c(\phi) = 0$ whenever $\phi$ is unramified.
\end{itemize}

Furthermore, a counting function is called \textit{proper} if it satisfies the following condition:
Let $p$ and $p'$ be any two primes not dividing $|\Gamma|$, and let $I_{\Q_p}$ and $I_{\Q_{p'}}$ be the absolute inertia groups of $\Q_p$ and $\Q_{p'}$, respectively.
If $\phi:G_{\Q_p} \to \Gamma$ and $\phi':G_{\Q_{p'}} \to \Gamma$ with $\phi(I_{\Q_p}) = \phi'(I_{\Q_{p'}})$, then $c(\phi) = c(\phi')$.
That is, for tame primes, $c$ depends only on the image of the inertia group.

We follow Wood's notation in \cite{MMW-MF} here, except that we allow $c$ to take values in $\R$.
If $c$ takes only values in $\Z_{\ge 0}$, we call it \textit{natural}.

Also as in \cite{MMW-MF},
we define the \textit{total mass at} $p$ of a counting function $c$ as
\[
 M(\Q_p,\Gamma,c) = \frac{1}{|\Gamma|} \sum_{\phi \in S_{\Q_p,\Gamma}} \frac{1}{p^{c(\phi)}}.
\]
Note that this sum is finite, so the right-hand side is well-defined.
Kedlaya \cite{Kedlaya-CF} and Wood \cite{MMW-MF} omit the factor of $\frac{1}{|\Gamma|}$, but we include it for simpilicity later.  We will see in Theorem \ref{thm:mf-coeffs} that for natural counting functions, all the coefficients of the Laurent polynomial $M(\Q_p,\Gamma,c)$ are still integers after dividing by $|\Gamma|$.

\subsection{Mass Formulas}
Define a \textit{character Laurent polynomial} to be a sum
\[
 f(x) = \sum_{i = k_1}^{k_2} \sigma_i(x) x^{-i},
\]
defined for integers $x$, where $k_1,k_2 \in \Z$, and each $\sigma_i$ is a $\Z$-linear combination of \textit{minimal} Dirichlet characters modulo divisors of $|\Gamma|$.
Note that $i$ may take negative values if $k_1 < 0$.

We use the convention that if $\chi$ is a character with modulus $n$ and $(x,n) > 1$, then $\chi(x) = 0$, and we define a \textit{minimal} character as follows:
If $\chi$ is a Dirichlet character of modulus $n$, and there is no other character $\chi'$ with modulus $m < n$ such that $\chi(a) = \chi'(a)$ whenever $a$ is coprime to $mn$, then $\chi$ is minimal.
For example, consider the character $\chi$, with modulus 5, defined as $\chi(x) = 1$ when $5 \nmid x$ and $\chi(x) = 0$ when $5 |x$.  If we define $\chi'(x) = 1$ for all $x$, then $\chi(x) = \chi'(x)$ for all $x$ not divisible by 5, so $\chi$ is not minimal.  All Dirichlet characters in this paper are henceforth assumed to be minimal.


\begin{defn}
If $f$ is a character Laurent polynomial, $c$ is a counting function for $\Gamma$, and $S$ is a set of primes, we say $f$ is an \textit{$S$-mass formula} for $c$ (and that $c$ has a mass formula for $S$) if for all primes $p \in S$, we have
\[
 M(\Q_p,\Gamma,c) = f(p).
\]
\end{defn}

If $S$ is the set of all primes in $\Z$, then we call the mass formula \textit{universal}.  If $S$ contains all primes not dividing $|\Gamma|$, we call the mass formula \textit{tame}.

\begin{defn}
If $f$ is a mass formula in which only the trivial Dirichlet character appears (i.e. $f$ is a Laurent polynomial with integer coefficients), then we call $f$ a \textit{pure mass formula}.
\end{defn}
Our ``pure mass formula'' corresponds to the definition of ``mass formula'' used by Kedlaya and Wood, except that we allow powers of $p$ other than negative integers to appear in $f$, accounting for non-natural counting functions.
The broader definition of a mass formula used here is necessary for the elegant result on tame mass formulas in Theorem \ref{thm:mf-coeffs}.


\begin{ex}
Let $\Gamma = C_2$.  Then each surjective $\phi \in \bigcup_p S_{\Q_p,\Gamma}$ corresponds to a distinct quadratic extension of $\Q_p$.
Define a counting function $c$ so that $c(\phi)$ is the discriminant exponent (the $p$-adic valuation of the discriminant) of this extension.

This counting function is proper, and it has a universal pure mass formula, as we can verify by computing masses explicitly using \cite{Jones-DB}.
If $p \neq 2$, there are two ramified quadratic extensions of $\Q_p$, each with discriminant exponent 1.
In addition, there is one unramified quadratic extension, and one non-surjective map $G_{\Q_p} \to C_2$ (the trivial map, which is also unramified), so the mass at $p$ is $1 + p^{-1}$.
For $p=2$, there are again two unramified maps $G_{\Q_p} \to C_2$, but now there are two quadratic extensions of $\Q_2$ with discriminant exponent 2, and four quadratic extensions with discrimiant exponent 3.
The mass at 2 is thus $1 + 2^{-2} + 2 \cdot 2^{-3} = 1 + 2^{-1}$.
Since this agrees numerically with the mass at all other primes, the mass formula $f(p) = 1 + p^{-1}$ is universal.
\end{ex}

The following two propositions are extensions of results due to Kedlaya \cite[Corollaries 5.4-5.5]{Kedlaya-CF}:

\begin{prop}
\label{prop:Kedlaya-cong}
Let $a$ be an integer relatively prime to $|\Gamma|$, and let $c$ be any proper counting function for $\Gamma$.  Then $c$ has a pure $S$-mass formula, where $S$ is the set of all primes congruent to $a$ modulo $|\Gamma|$.
\end{prop}

\begin{prop}
\label{prop:Kedlaya-tame}
Let $c$ be any proper counting function for $\Gamma$.  Then $c$ has a pure tame mass formula if and only if $\Gamma$ has a rational character table.
\end{prop}

Kedlaya only considers natural counting functions (using our terminology), so we will prove that Proposition \ref{prop:Kedlaya-cong} extends to non-natural counting functions as well.

\begin{proof}
Let $p$ be a prime congruent to $a$ modulo $|\Gamma|$.  Consider the quotient $G_{\Q_p} / G_{1,\Q_p}$, for $p \nmid |\Gamma|$,
where the latter group is the absolute wild inertia group.  This quotient is a semidirect product of $G_{0,\Q_p}/G_{1,\Q_p}$, the absolute tame inertia group, with $\hat{\Z}$.
Let the topological generators of $G_{0,\Q_p}/G_{1,\Q_p}$ and $\hat{\Z}$ be $s$ and $t$, respectively.
If a continuous homomorphism $\phi:G_{\Q_p} \to \Gamma$ is tamely ramified, then it factors through $G_{\Q_p} / G_{1,\Q_p}$, so we can regard $\phi$ as being completely determined by the images $\sigma$ and $\tau$ of $s$ and $t$, respectively,
and these choices must be satisfy $\tau \sigma \tau^{-1} = \sigma^p$.
Furthermore, if $c$ is a proper counting function, then $c(\phi)$ is determined only by the choice of $\sigma$.

Now suppose $q$ is another prime with $q = p + b \cdot |\Gamma|$, where $b \in \Z$.
Then for any $\sigma \in \Gamma$, we have $\sigma^q = \sigma^p \cdot \sigma^{b \cdot |\Gamma|} = \sigma^p$.
This shows that the pairs $(\sigma,\tau)$ with $\sigma,\tau \in \Gamma$ and $\tau \sigma \tau^{-1} = \sigma^p$ are the same as those with $\tau \sigma \tau^{-1} = \sigma^q$,
and thus there is a one-to-one correspondence between tamely ramified maps in $S_{\Q_p,\Gamma}$ and $S_{\Q_q,\Gamma}$, which preserves the value of any proper counting function $c$.

From this, it follows that the total masses of $c$ at $p$ and $q$ are the same Laurent polynomial in $p$ and $q$, and thus $c$ has a pure mass formula for the set of all primes congruent to $a$ modulo $|\Gamma|$.
\end{proof}

We omit a similar proof for Proposition \ref{prop:Kedlaya-tame}, since we will later generalize this result for non-pure mass formulas.

\begin{remark}
In this paper, we usually discuss global maps $\phi:G_\Q \to \Gamma$, and their restrictions to $G_{\Q_p}$.  However, one could instead consider global extensions $K/\Q$ with Galois group $\Gamma$, and their completions $K_p/\Q_p$ at primes above $p$.  A map from $G_\Q$ to $\Gamma$ is equivalent to a global extension $K/\Q$ along with a choice of isomorphism between $\gal(K/\Q)$ and $\Gamma$, so one can formulate the definition of a mass formula in terms of global field extensions provided the total mass includes a factor of $\frac{1}{|\operatorname{Aut}(K_p)|}$.
	
We use the term ``$\Gamma$-extension'' to refer to either the map or the field extension, taking the isomorphism $\gal(K/\Q) \to \Gamma$ to be implicit.

By way of notation, if $K/\Q$ is a $\Gamma$-extension and $H$ is a subgroup of $\Gamma$, then $K_H$ will denote the fixed field of $H$ in $K$.
\end{remark}

\section{Weighted discriminants}
\label{sec:wdisc}

We use the term \textit{alternate discriminant} to refer to a real-valued function on the set of $\Gamma$-extensions (in the sense of maps $\phi:G_\Q \to \Gamma$) of $\Q$.
A ``reasonable'' alternate discriminant, broadly speaking, should take positive integer values, and its valuation at each rational prime $p$ should be determined by the restriction of $\phi$ to $G_{\Q_p}$.

If we require alternate discriminants to be ``determined locally'' in this way, then there is a natural bijection between alternate discriminants and counting functions.
Given a counting function $c$ for $\Gamma$, we can build an alternate discriminant corresponding to $c$ as follows:
Let $\phi : G_\Q \to \Gamma$ be a $\Gamma$-extension.
Then letting $\phi_p$ be the restriction of $\phi$ to $G_{\Q_p}$, we define
\[
 D_c(\phi) = \prod_p p^{c(\phi_p)}.
\]
Conversely, if an alternate discriminant $D$ is determined locally, then we can construct a counting function $c$ corresponding to it.
If $\phi_p$ is the restriction of some $\Gamma$-extension to $G_{\Q_p}$, then $c(\phi_p)$ is the power of $p$ appearing in $D(\phi)$.

However, from the perspective of searching for universal mass formulas, this broad class of invariants is not very interesting, even if we require alternate discriminants to be determined locally and counting functions to be proper.
As we will see in Theorem \ref{thm:mf-coeffs}, any proper counting function $c$ is guaranteed to have a tame mass formula.
Then, since the condition of properness imposes no restrictions on how the counting function can behave at primes dividing $|\Gamma|$,
we can assign values of $c$ in such a way that it forces the tame mass formula to be universal.

Thus, we seek a natural way to define counting functions (or alternate discriminants) globally, and prohibit entirely contrived behavior at wild primes.
To that end, in this paper we consider \textit{weighted discriminants}, a class of alternate discriminants generalizing Wood's work in \cite{MMW-MF}.
A weighted discriminant for $K/\Q$ is built by taking the discriminants of intermediate extensions, norming down to $\Q$, and raising each to a power (the \textit{weight}) associated to that intermediate extension.  Specifically:

\begin{defn}
	A \textit{weight function} for $\Gamma$ is a function
	$w: \{ (H,H') \} \to \R$,
	where the domain of $w$ is the set of ordered pairs $(H,H')$, where $H$ is a subgroup of $\Gamma$ and $H'$ is a maximal subgroup of $H$.
	
	The \textit{weighted discriminant} given by a weight function $w$ is
	\[
	D_{w}(K) = \prod_{(H,H')} N_{K_H/\Q} (\disc(K_{H'} / K_H))^{w(H,H')},
	\]
	where $\disc$ is the standard relative discriminant and $N$ is the norm.
\end{defn}

Since $N_{K_H/\Q} (\disc(K_{H'} / K_H))^{w(H,H')}$ can be determined locally from the ramification groups of $K/\Q$,
$D_w$ is an alternate discriminant that can be defined in terms of a counting function $\cw$.
We call a counting function of this form a \textit{weighted discriminant counting function}.

If $w(H,H') \in \Z_{\ge 0}$ for each $(H,H')$, we call $w$ \textit{positive integral}.
If $w$ is positive integral, then its counting function $\cw$ is natural, but the converse need not hold.
See Section \ref{subsec:mfq8} for an example of a non-integer-valued weight function whose counting function is nonetheless natural.

\begin{remark}
When computing $D_{w}(K)$, we take as implicit an isomorphism between $\gal(K/\Q)$ and $\Gamma$.  Changing this isomorphism by composing with an outer automorphism of $\Gamma$ may change the value of $D_{w}(K)$, but composing with an inner automorphism will not.
\end{remark}


\begin{remark}
It is possible for two different weight functions to give the same counting function.
For example, let $\Gamma = C_2 \times C_2$, and let $H_1$, $H_2$, and $H_3$ be its order-2 subgroups, with $1$ denoting the trivial subgroup.
If we let $w$ be the weight function with $w(\Gamma,H_1) = 2$ and all other weights equal to 0, and $w'$ be the weight function with $w(H_2,1) = w(H_3,1) = 1$ and all other weights zero,
then $c_w = c_w'$.
\end{remark}

\section{An Explicit Formula for $\cw$}
\label{sec:wcf}

In this section, we give an explicit formula for $\cw(\phi)$ in terms of the weight function $w$ and the ramification groups of the map $\phi$,
which we will use in the proof of Theorem \ref{thm:finiteMF}.

Let $\phi:G_\Q \to \Gamma$ be a continuous homomorphism, and let $K/\Q$ be the corresponding field extension.
If $\pa$ is a prime of $K$ above $p$, we denote by $I_{\pa,i}$ the $i$th ramification group in lower numbering at $\pa$, for the extension $K/\Q$.
As in \cite{serre-LF}, $i=0$ and $i=-1$ correspond to the inertia and decomposition groups, respectively.
Throughout this section, $\disc$ denotes the standard discriminant ideal, and $\diff$ the different ideal.

Let $H \le \Gamma$, and $H'$ be a maximal subgroup of $H$.  Recall that $K_H$ and $K_{H'}$ denote the fixed fields of $H$ and $H'$ in $K$.

Using the fact that the discriminant of a field extension is the norm of the different ideal, and that
\[
 \disc K/K_H = N_{K_{H'}/K_H} ( \disc K/H') \cdot (\disc K_{H'}/K_H)^{|H'|},
\]
we first obtain
\[
 \disc K_{H'}/K_H = \left( \frac{N_{K/K_H} \diff(K/K_H)}{N_{K/K_H} \diff(K/K_{H'})} \right)^{\frac{1}{|H'|}}.
\]
Norming down to $\Q$ gives:
\[
 N_{K_H/\Q} (\disc(K_{H'} / K_H)) = N_{K/\Q} \left( \frac{\diff(K/K_H)}{\diff(K/K_{H'})} \right)^{\frac{1}{|H'|}}.
\]
Now we take the valuation at $p$ of both sides, and use the fact that if $\pa$ is a prime above $p$ and $K/\Q$ is Galois,
then $N_{K/\Q}(\pa) = p^{f_{K/\Q}(p)}$, where $f$ denotes the degree of the residue field extension $K_\pa/\Q_p$.  This yields:
\[
 v_p(N_{K_H/\Q} (\disc(K_{H'} / K_H))) = \frac{f_{K/\Q}(p)}{|H'|} \sum_{\pa | p} \left( v_\pa (\diff(K/H)) - v_\pa (\diff(K/H')) \right).
\]
Using the formula in \cite{serre-LF} for the different in terms of the ramification groups of an extension,
and that $f_{K/\Q}(p) = \frac{|I_{p,-1}|}{|I_{p,0}|}$, the right side becomes
\[
 \frac{|I_{p,-1}|}{|I_{p,0}| \cdot |H'|} \sum_{\pa | p} \left[ \sum_{i \ge 0} \left( |I_{\pa,i} \cap H| - |I_{\pa,i} \cap H'| \right) \right].
\]

Now choose any prime $\pa$ above $p$.
The ramification groups of the other primes above $p$ are conjugates of $I_{\pa,i}$.
There are $|\Gamma|/|I_{\pa,-1}|$ of these, so we can rewrite the previous line as
\[
 \frac{|I_{\pa,-1}|}{|I_{\pa,0}| \cdot |H'|} \cdot \frac{1}{|I_{\pa,-1}|} \sum_{\gamma \in \Gamma} \left[ \sum_{i \ge 0} \left(
|\gamma I_{\pa,i} \gamma^{-1} \cap H| - |\gamma I_{\pa,i} \gamma^{-1} \cap H'| \right) \right].
\]
If $\phi : G_\Q \to \Gamma$ is a continuous homomorphism with inertia groups $I_{p,i}$, and $\phi_p$ is its restriction to $G_{\Q_p}$, then we set
\[
 c_{H,H'}(\phi_p) := \frac{1}{|I_{\pa,0}| \cdot |H'|} \cdot \sum_{\gamma \in \Gamma} \left[ \sum_{i \ge 0} \left(
|\gamma I_{\pa,i} \gamma^{-1} \cap H| - |\gamma I_{\pa,i} \gamma^{-1} \cap H'| \right) \right].
\]
This expression does not depend on the choice of $\pa$, since we sum over all conjugates of $I_{\pa,i}$.

Now for any weight function $w$ with corresponding weighted discriminant $D_{w}$, define a counting function $\cw$ by
\[
 \cw(\phi_p) = \sum_{(H,H')} c_{H,H'}(\phi_p) \cdot w(H,H').
\]
Let $\phi:G_\Q \to \Gamma$, with $\phi_p$ the restriction of $\phi$ to $G_{\Q_p}$.  Since
\[
 \cw(\phi_p) = \sum_{(H,H')} v_p(N_{K_H/\Q} (\disc(K_{H'} / K_H))) \cdot w(H,H'),
\]
we have
\begin{align*}
D_{\cw}(K) &= \prod_p p^{\cw(\phi_p)} \\
& = \prod_p \prod_{(H,H')} p^{v_p(N_{K_H/\Q} (\disc(K_{H'} / K_H))) \cdot w(H,H')} \\
& = \prod_{(H,H')} N_{K_H/\Q} (\disc(K_{H'} / K_H))^{w(H,H')} \\
& = D_w(K).
\end{align*}

Thus if $c_{H,H'}$ and $\cw$ are defined as above, then $\cw$ is the counting function corresponding to the weighted discriminant $D_w$.

If $p \nmid |\Gamma|$, then $\cw$ depends only on the inertia groups $I_{\pa,0}$, and in particular not on the decomposition group.
This implies:

\begin{cor}
\label{cor:wts-cf}
Given any weight function $w$, the corresponding counting function $\cw$ is proper.
\end{cor}

\section{The Overall Weight of a Subgroup}
\label{sec:overall-wt}

It is possible for several weight functions to give the same weighted discriminant counting function.
We can now use the explicit formula in Section \ref{sec:wcf} to determine when this happens, and in fact show that such a counting function actually depends on a smaller set of parameters, which we call the \textit{overall weights}.

\begin{defn}
	Let $w$ be a weight function for $\Gamma$, and let $I \subseteq \Gamma$ be any subgroup of $\Gamma$.  The \textit{overall weight} of $I$ is the quantity
	\begin{equation}
		\label{eqn:overall-defn}
		\ow(I) = \sum_{(H,H')} \left( \frac{w(H,H')}{|I| \cdot |H'|} \sum_{\gamma \in \Gamma} \left( |\gamma I \gamma^{-1} \cap H| - |\gamma I \gamma^{-1} \cap H'| \right) \right),
	\end{equation}
	where the first sum, as usual, ranges over pairs where $H \subseteq \Gamma$ and $H'$ is a maximal subgroup of $H$.
\end{defn}

Note that if $I$ and $I'$ are conjugate in $\Gamma$, then $\ow(I) = \ow(I')$, so it makes sense to speak of the overall weight of a conjugacy class of subgroups.
More importantly, if $\phi_p:G_{\Q_p} \to \Gamma$ with ramification groups $I_{p,i}$, then the explicit formula in Section \ref{sec:wcf} reduces to
\[
\cw(\phi_p) = \sum_{i \ge 0} \frac{|I_{p,i}|}{|I_{p,0}|} \ow(I_{p,i}).
\]
This shows that the counting function attached to $w$ depends only on the overall weights.
That is, if $w_1$ and $w_2$ are weight functions and $\ow_1(I) = \ow_2(I)$ for every subgroup $I \subseteq \Gamma$, then $c_{w_1} = c_{w_2}$ (and thus $D_{w_1} = D_{w_2}$).

In fact, we can go even further:

\begin{prop}
	\label{prop:overall-only}
	If $w$ is a weight function for $\Gamma$, then the overall weights of noncyclic subgroups of $\Gamma$ can be expressed in terms of the overall weights of cyclic subgroups.  In particular, $c_w$ and $D_w$ are completely determined by the overall weights of cyclic subgroups.
\end{prop}

\begin{proof}
	Consider the quantity $|\gamma I \gamma^{-1} \cap H|$ in the formula for the overall weight.  Since $I$ is the union of its cyclic subgroups, we have
	\[
	\gamma I \gamma^{-1} \cap H = \bigcup_{C \subseteq I} \gamma C \gamma^{-1} \cap H,
	\]
	where the sum ranges over all cyclic subgroups of $I$.  By the principle of inclusion-exclusion, it follows that
	\[
	|\gamma I \gamma^{-1} \cap H| = \sum_{n \ge 1} \left( (-1)^{n+1} \sum_{C_1,\ldots,C_n \subseteq I} |\gamma (C_1 \cap \ldots \cap C_n) \gamma^{-1} \cap H| \right),
	\]
	where the inner sum ranges over all unordered $n$-tuples $(C_1, \ldots, C_n)$ of distinct cyclic subgroups of $I$.  For ease of notation, if $\mathcal{C}$ is any such unordered $n$-tuple, let $|\gamma (C_1 \cap \ldots \cap C_n) \gamma^{-1} \cap H| = u(\gamma, H, \mathcal{C})$, and note that the formula above does not require $H$ to be a subgroup of $\Gamma$, only a subset.
	
	We can now express the overall weight of $I$ as
	\[
	\ow(I) = \sum_{(H,H')} \left( \frac{w(H,H')}{|I| \cdot |H'|} \sum_{\gamma \in \Gamma} \sum_{n \ge 1}   \sum_{\mathcal{C}} (-1)^{n+1} u(\gamma, H \setminus H', \mathcal{C})  \right),
	\]
	where the innermost sum is the same as in the previous line.
	Then we can rearrange the sums to obtain
	\[
	\ow(I) = \sum_{n \ge 1} \, \sum_{\mathcal{C}} \, \sum_{(H,H')} \left( \frac{w(H,H')}{|I| \cdot |H'|} \sum_{\gamma \in \Gamma}  (-1)^{n+1} u(\gamma, H \setminus H', \mathcal{C}) \right).
	\]
	Finally, if $\mathcal{C} = (C_1, \ldots, C_n)$, then
	\[
	\ow(C_1 \cap \ldots \cap C_n) = \sum_{(H,H')} \left( \frac{w(H,H')}{|C_1 \cap \ldots \cap C_n| \cdot |H'|} \sum_{\gamma \in \Gamma}  u(\gamma, H \setminus H', \mathcal{C})  \right),
	\]
	and therefore
	\begin{equation}
		\label{eqn:overall-cyclic} 
		\ow(I) = \sum_{n \ge 1} \, \sum_{C_1,\ldots,C_n \subseteq I} \, \left( (-1)^{n+1} \frac{|C_1 \cap \ldots \cap C_n|}{|I|} \cdot \ow \left( C_1 \cap \ldots \cap C_n \right) \right).
	\end{equation}
	
	Since $C_1 \cap \ldots \cap C_n$ is itself cyclic, this expresses $\ow(I)$ in terms of the overall weights of cyclic groups, as desired.
\end{proof}

A converse of Proposition \ref{prop:overall-only} also holds:

\begin{prop}
	\label{prop:any-overall}
	For any choice of one real number for each conjugacy class of nontrivial cyclic subgroups of $\Gamma$, there is a weight function $w$ such that for any nontrivial cyclic subgroup $C \subseteq \Gamma$,
	$\ow(C)$ is the real number assigned to the conjugacy class of $C$.
\end{prop}

The proof of Proposition \ref{prop:any-overall} is long and technical, so we postpone it to Section \ref{sec:any-overall-proof}.

These two results together imply:

\begin{cor}
	\label{cor:dimension-of-cfs}
	The set of (not necessarily natural) weighted discriminant counting functions for a finite group $\Gamma$ can be viewed as a vector space over $\R$,
	with basis vectors corresponding to the conjugacy classes of cyclic subgroups of $\Gamma$.
\end{cor}

This suggests that we should think of a weighted discriminant counting function as being determined by the choice of overall weights $\ow(C)$, rather than the weights $w(H,H')$.
The overall weights are also related to whether or not $\cw$ is natural:

\begin{prop}
	\label{prop:overall-natural}
	If $\cw$ is natural, then all the overall weights of cyclic subgroups given by $w$ are nonnegative integers.
\end{prop}

\begin{proof}
	First, suppose that $\cw$ is natural.
	Let $C$ be any cyclic subgroup of $\Gamma$.
	Recall from the proof of Proposition \ref{prop:Kedlaya-tame} that if $\ell$ is a prime not dividing $|\Gamma|$, then maps $G_{\Q_\ell} \to \Gamma$ are determined by a pair $(s,t)$ in $\Gamma$ with $tst^{-1} = s^\ell$.
	If $\ell \equiv 1 \bmod |\Gamma|$, then $s^\ell = s$.
	By letting $s$ be a generator of $C$ and $t = 1_\Gamma$, we can construct a tamely ramified map $\phi:G_{\Q_\ell} \to \Gamma$ in which the image of inertia is $C$.
	Then $\cw(\phi) = \ow(C)$, so $\ow(C)$ must be a nonnegative integer.
\end{proof}

Looking at counting functions through the lens of overall weights reveals one more fact, which will be useful in the proof of Theorem \ref{thm:finiteMF}:

\begin{cor}
	\label{cor:natural-double}
	Let $\phi_p:G_{\Q_p} \to \Gamma$ and $\phi'_\ell:G_{\Q_\ell} \to \Gamma$, with $p \neq \ell$,
	and let $I_{\ell,i}(\phi'_\ell)$ and $I_{p,i}(\phi_p)$ denote the $i$th ramification groups of $\phi'_\ell$ and $\phi_p$.  If
	\begin{itemize}
		\item $\Gamma$ is an $\ell$-group,
		\item $c$ is a natural weighted discriminant counting function for $\Gamma$, and
		\item $I_{\ell,0}(\phi'_\ell) = I_{p,0}(\phi_p)$,
	\end{itemize}
	then $\cw(\phi'_\ell) \ge 2 \cw(\phi_p)$.
\end{cor}

\begin{proof}
	Since $\phi_p$ is tamely ramified, the inertia group $I_{p,0}(\phi_p)$ is cyclic.
	Since $\Gamma$ is an $\ell$-group, we have $I_{\ell,0}(\phi'_\ell) = I_{\ell,1}(\phi'_\ell) = I_{p,0}(\phi_p)$.
	Furthermore, all of the higher ramification groups $I_{\ell,i}(\phi'_\ell)$ must be cyclic as well, so their overall weights are nonnegative, by Proposition \ref{prop:overall-natural}.
	Thus
	\begin{align*}
		\cw(\phi'_\ell) &= \sum_{i \ge 0} \frac{|I_{\ell,i}(\phi'_\ell)|}{|I_{\ell,0}(\phi'_\ell)|} \ow(I_{\ell,i}(\phi'_\ell)) \\
		& \ge \sum_{i = 0}^1 \frac{|I_{\ell,i}(\phi'_\ell)|}{|I_{\ell,0}(\phi'_\ell)|} \ow(I_{\ell,i}(\phi'_\ell)) \\
		& = 2 \ow \left( I_{\ell,0}(\phi'_\ell) \right) \\
		& = 2 \ow \left( I_{p,0}(\phi_p) \right) \\
		& = 2 \cw(\phi_p).
	\end{align*}
\end{proof}

\begin{remark}
	The overall weight of a subgroup generalizes Malle's \textit{index} in \cite{malle:distro2}.
	In particular, if $w$ is the weight function corresponding to the standard discriminant, then the overall weight of each subgroup is
	equal to the index of its generator, viewing $\Gamma$ via its regular representation.
\end{remark}

\section{Tame Mass Formulas and Their Coefficients}
\label{sec:bridge}

In this section, we prove a more general form of Proposition \ref{prop:Kedlaya-tame} for non-pure mass formulas:

\begin{theorem}
	\label{thm:mf-coeffs}
	Any proper natural counting function $c$ has exactly one tame mass formula.
	The tame mass formula is of the form
	\[
	f(x) = \sum_C \sigma_C(x) x^{-i_C},
	\]
	where the sum ranges over conjugacy classes of cyclic subgroups $C \subseteq \Gamma$.
	Each ``coefficient'' $\sigma_C$ is a sum of distinct Dirichlet characters modulo divisors of $|\Gamma|$, one of which is the trivial character.
	
	Furthermore, if $c$ is a weighted discriminant counting function with weight function $w$, then $i_C = \ow(C)$.
\end{theorem}

\begin{remark}
	Proposition \ref{prop:Kedlaya-tame}, in this context, implies that the mass formula given by Theorem \ref{thm:mf-coeffs} is pure if and only if $\Gamma$ has a rational character table.
\end{remark}

We will need the following fact from representation theory:
\begin{prop}
\label{prop:char-sums}
Let $A$ be an abelian group, and $B$ a subgroup of $A$.  Let $\Sigma$ be the sum of all irreducible characters of $A$ that are trivial on $B$.  Then
\[
 \Sigma(a) = \begin{cases} 0 & \text{if } a \notin B \\ [A:B] & \text{if } a \in B. \end{cases}
\]
\end{prop}
Briefly, this statement follows from the fact that characters of $A$ trivial on $B$ correspond to characters of $A/B$, and summing over all such characters gives $|A/B|$ on $1_{A/B}$ and 0 elsewhere.

Also, we use the notation $g_1 \sim g_2$ to mean that $g_1,g_2 \in \Gamma$ are conjugates, and $[g]$ to denote the conjugacy class of $g$ in $\Gamma$.

We now prove Theorem \ref{thm:mf-coeffs}.

\begin{proof}
Let $a$ be an integer relatively prime to $|\Gamma|$.
Since $c$ is proper, there exists a pure mass formula $f_a$ for all primes congruent to $a$ modulo $|\Gamma|$, by Proposition \ref{prop:Kedlaya-cong}.
There can be only one such pure mass formula, since if $f_a'$ were another, then $f_a$ and $f_a'$ would be two different Laurent polynomials which agree at infinitely many values,
which is impossible.

For convenience, we will assume from this point on that $c$ is a weighted discriminant counting function with weight function $w$.
In the general case, we are dealing only with tame ramification (and thus cyclic ramification groups) and $c$ is proper.  Thus $c$ depends only on a quantity similar to the overall weight for each conjugacy class of cyclic subgroups.  The argument is then the same, but with the overall weights $\ow(C)$ replaced by these other quantities.

From the argument in the proof of Proposition \ref{prop:Kedlaya-cong}, we can see that
\[
f_a(x) = \sum_C n_C x^{-\ow(C)},
\]
where the sum runs over conjugacy classes of cyclic subgroups of $\Gamma$, and
$n_C$ is the number of maps $G_{\Q_p} \to \Gamma$ with $p \equiv a \bmod |\Gamma|$ whose inertia group is conjugate to $C$.

Now let $f = \sum \sigma_i x^i$, a character Laurent polynomial, be a tame mass formula for $c$.  
If $p \equiv a \bmod |\Gamma|$, then we must have $f(p) = f_a(p)$, so the only exponents appearing in the sum are $-\ow(C)$ for cyclic subgroups $C \subseteq \Gamma$.  That is, $f$ is of the form
\[
f(x) = \sum_C \left( \sum_{\chi_j} b_{C,j} \chi_j(x) \right) x^{-\ow(C)},
\]
where the inner sum runs over all Dirichlet characters modulo divisors of $|\Gamma|$.
If $p$ is sufficiently large compared to all the $b_{C,j}$ and $n_C$, then letting $x=p$, it follows that for every $C$, we must have
\begin{equation}
	\sum_{\{C': \ow(C') = \ow(C)\}} \sum_{\chi_j} b_{C',j} \chi_j(p) = \sum_{\{C': \ow(C') = \ow(C)\}} n_{C'}.
\end{equation}


Since the $\chi_j$ are periodic, this must in fact hold for all $p \equiv a \bmod |\Gamma|$.
That is, the value of each ``coefficient'' $\sigma_i$ in $f$ on each $a \in (\Z/|\Gamma|\Z)^*$ is determined by the corresponding coefficient of $f_a$.

Since $(\Z/|\Gamma|\Z)^*$ is abelian, its conjugacy classes each contain a single element, so each $\sigma_i$ is in fact a class function on $(\Z/|\Gamma|\Z)^*$.
Thus $\sigma_i$ is uniquely expressible as a $\C$-linear combination of irreducible characters of $(\Z/|\Gamma|\Z)^*$.
Finally, each irreducible character of $(\Z/|\Gamma|\Z)^*$ is equal to a unique (minimal) Dirichlet character whose modulus divides $|\Gamma|$.


This shows that $f$ can be expressed uniquely in the form
\[
f(x) = \sum_C \sigma_C(x) x^{-\ow(C)},
\]
where each $\sigma_C$ is a complex linear combination of Dirichlet characters modulo divisors of $|\Gamma|$.

To complete the proof of Theorem \ref{thm:mf-coeffs}, we now show that that each coefficient $\sigma_C$ in $f$ is a sum of distinct characters, one of which is the trivial character.

Let $p$ be a prime not dividing $|\Gamma|$, and let $f_p$ be the pure mass formula for the set of primes congruent to $p$ modulo $|\Gamma|$,
as discussed above.  As described above, $f_p$ also has a term corresponding to each conjugacy class of cyclic subgroups of $\Gamma$, although the exponents of these terms are not necessarily distinct.

Let $\gamma$ be an element of $\Gamma$.
The coefficient $\sigma_{\langle \gamma \rangle}(p)$ is $\frac{1}{|\Gamma|}$ times the number of maps $G_{\Q_p} \to \Gamma$ with inertia group conjugate to $\langle \gamma \rangle$.
Each such map is specified by an ordered pair $(s,t) \in \Gamma^2$, where $\langle t \rangle$ is conjugate to $\langle \gamma \rangle$, and $sts^{-1} = t^p$.
(In the language of number fields, $t$ is the generator of inertia, and $s$ is the Frobenius element.)

If $\gamma^p \notin [\gamma]$, then $t^p \notin [t]$ for any $t$ with $\langle t \rangle$ conjugate to $\langle \gamma \rangle$, so there are no such pairs.
Otherwise, the number of choices for $t$ is the number of elements of $\Gamma$ generating a subgroup conjugate to $\langle \gamma \rangle$,
and the number of choices for $s$ is equal to the number of elements of $\cent(\gamma)$, the centralizer of $\gamma$ in $\Gamma$.

In the latter case, let $n = \operatorname{ord}_\Gamma (\gamma)$.
Suppose $\gamma \sim \gamma^a$ and $\gamma \sim \gamma^b$ with $a$ and $b$ are coprime to $n$.  Choose $g_1,g_2 \in \Gamma$ with $g_1 \gamma g_1^{-1} = \gamma^a$, and $g_2 \gamma g_2^{-1} = \gamma^b$.  Then
\[
g_2 g_1 \gamma g_1^{-1} g_2^{-1} = g_2 \gamma^a g_2^{-1} = \gamma^{ab}.
\]
Thus $[\gamma] \cap \langle \gamma \rangle$ is naturally in bijection with a subgroup $S \subseteq (\Z / n \Z)^*$, via $\gamma^k \mapsto k$.

We can now calculate $\sigma_{\langle \gamma \rangle}(p)$.
For each element of $[\gamma]$, we have one choice for $t$, but we also need to count elements of $\Gamma$ not in $[\gamma]$ but generating a subgroup conjugate to $\langle \gamma \rangle$.
Overall, then, a choice of $t$ is described by a choice of an element of $[\gamma]$ and a coset of $S$ in $(\Z / n \Z)^*$.
The number of choices for $s$, as above, is $|\cent(\gamma)|$.
It follows that 
\[
\sigma_{\langle \gamma \rangle}(p) = \frac{1}{|\Gamma|} \cdot |[\gamma]| \cdot \frac{|(\Z/n\Z)^*|}{|S|} \cdot |\cent(\gamma)| = \frac{|(\Z/n\Z)^*|}{|S|} = \left[ (\Z / n\Z)^* : S \right],
\]
since $|[\gamma]| \cdot |\cent(\gamma)| = |\Gamma|$.

Now, we have $\gamma \sim \gamma^p$ if and only if $p \in S$, when $p$ is taken as an element of $(\Z / n\Z)^*$.
Thus $\sigma_{\langle \gamma \rangle}(p)$ should be 0 if $p \notin S$, and $\left[ (\Z / n\Z)^* : S \right]$ if $p \in S$.

Let $\Sigma_{n,S}$ be the sum of all irreducible characters of $(\Z / n\Z)^*$ that are trivial on $S$.
By Proposition \ref{prop:char-sums}, $\sigma_{\langle \gamma \rangle}(p) = \Sigma_{n,S}(p)$.
Thus $\Sigma_{n,S}$ is equal to $\sigma_{\langle \gamma \rangle}$, the ``coefficient'' of $f$ corresponding to the conjugacy class of $\langle \gamma \rangle$.
Finally, $\Sigma_{n,S}$ is a sum of distinct Dirichlet characters including the trivial character, as desired.

\end{proof}

\section{Proof of Theorem \ref{thm:finiteMF}}
\label{sec:big-proof}

We now are equipped to prove our main theorem, Theorem \ref{thm:finiteMF}.
Let $\Gamma$ be an $\ell$-group, and $c$ a natural weighted discriminant counting function for $\Gamma$, with weight function $w$, and corresponding weighted discriminant $D_w$.
Assume that $c$ has a universal mass formula $f$.
Our method of proof will be to put an upper bound on each of the overall weights given by $w$.
Since $c$ is natural (which implies all the overall weights are nonnegative integers), this fact, combined with Propositions \ref{prop:overall-only} and \ref{prop:overall-natural}, will show that there are only finitely many choices for the overall weights and thus for $c$.

\subsection{Preliminaries}
\label{subsec:prelims}

If $f$ is universal, it must be exactly the unique tame mass formula described in Theorem \ref{thm:mf-coeffs}.
Let $[C_1],\ldots,[C_s]$ be the conjugacy classes of cyclic subgroups of $\Gamma$.
By Theorem \ref{thm:mf-coeffs}, $f$ is of the form
\[
f(p) = \sum_{[C_j]} \sigma_{C_j}(p) p^{-\ow(C_j)},
\]
where each $\sigma_{C_j}$ is a sum of Dirichlet characters, exactly one of which is trivial.
Since each nontrivial character vanishes at $\ell$, we have
\[
f(\ell) = \sum_{[C_j]} \ell^{-\ow(C_j)}.
\]
The mass formula $f$ is universal if and only if this quantity is equal to the total mass of $c$ at $\ell$.

Note that $f(\ell)$ need only be numerically equal to the total mass; the two quantities will never be abstractly the same poylnomial in $\ell$.
For example, if we take $\Gamma = C_2$, and $D_{w}$ to be the standard discriminant, then the tame mass formula is
\[
f(p) = 1 + p^{-1}.
\]
At $\ell=2$, there are two quadratic extensions of $\Q_2$ of discriminant 4 and four extensions of discriminant 8 \cite{Jones-DB}, so the total mass is
\[
1 + \ell^{-2} + 2 \ell^{-3}.
\]
However, since
\[
1 + 2^{-1} = 1 + 2^{-2} + 2 \cdot 2^{-3} = \frac{3}{2},
\]
the mass formula $1 + p^{-1}$ is universal.

\subsection{The Universality Equation}
\label{subsec:MFF}

Since no mass formula for $c$ other than $f$ (the one given by Theorem \ref{thm:mf-coeffs}) can be universal, $c$ has a universal mass formula if and only if these overall weights satisfy the ``universality equation''

\begin{equation}
	\label{eqn:MFF} \sum_{[C_j]} \ell^{-\ow(C_j)} = \sum_{\phi:G_{\Q_\ell} \to \Gamma} \ell^{-c(\phi)}.
\end{equation}
Note that by the arguments in Section \ref{sec:overall-wt}, each exponent on the right side is a linear combination (with rational coefficients) of the overall weights of cyclic subgroups.

Furthermore, since $c$ is completely determined by the overall weights it associates to cyclic subgroups (by Proposition \ref{prop:overall-only}),
a weighted discriminant counting function with a universal mass formula is equivalent to a choice of $(\ow(C_1),\ldots,\ow(C_s))$ satisfying this equation.
Once we have studied the right side of the universality equation, we then prove Theorem \ref{thm:finiteMF} by showing that it has finitely many nonnegative integer solutions for $(\ow(C_1),\ldots,\ow(C_s))$.

The following lemma helps make sense of the right-hand side of equation (\ref{eqn:MFF}):

\begin{lemma}
	\label{lemma:extn-exist}
	There exists a totally ramified cyclic extension of $\Q_\ell$ of degree $\ell^k$ for each $k$.
	Equivalently, there exists a surjective totally ramified map $\phi: G_{\Q_\ell} \to C_{\ell^k}$ for each $k$.
\end{lemma}

\begin{proof}
	We can construct such an extension using cyclotomic extensions of $\Q_\ell$.
	Adjoining a primitive $\ell^m$th root of unity gives a totally ramified extension with Galois group $(\Z/\ell^m\Z)^*$ \cite{serre-LF}.
	If $m$ is large enough, then $(\Z/\ell^m\Z)^*$ has a subgroup for which the quotient is isomorphic to $C_{\ell^k}$; by taking the corresponding subfield, we obtain the desired extension.
\end{proof}

If $C_j$ is any cyclic subgroup of $\Gamma$, and $\phi_j$ is the map in Lemma \ref{lemma:extn-exist}, then $c(\phi_j) \ge 2 \ow(C_j)$ by Corollary \ref{cor:natural-double}.
On the right side of equation (\ref{eqn:MFF}), we can split off all the maps obtained from Lemma \ref{lemma:extn-exist} to produce
\begin{align}	
	\sum_{[C_j]} \ell^{-\ow(C_j)} &= \sum_{[C_j]} \ell^{-c(\phi_j)} + \sum_{\text{other } \phi} \ell^{-c(\phi)} \\	\label{eqn:MFF2}
	&= \sum_{[C_j]} \ell^{-b_j \ow(C_j)} + \sum_{\text{other } \phi} \ell^{-c(\phi)},
\end{align}
where each $b_j$ is at least 2.

We will now put upper bounds on the overall weights $\ow(C_j)$ by studying the $\ell$-adic valuation of each side of equation (\ref{eqn:MFF2}).

\subsection{The Upper Bound}
\label{subsec:ubounds}

Let $M$ be the greatest of the $\ow(C_j)$, and let $t$ be the total number of terms on the right side of equation (\ref{eqn:MFF2}).

Each term on the right is a power of $\ell$, and one of them has $\ell$-adic valuation less than or equal to $-2M$.
Since each power of $\ell$ added to this term can increase the valuation by at most 1, the largest possible valuation of the right side is then $-2M + t - 1$.

Meanwhile, the valuation of the left side is greater than or equal to $-M$, since no term has a valuation smaller than this.
Thus for $f$ to be universal, we must have
\[
-M \le -2M + t - 1.
\]
This implies that $M \le t - 1$.
Thus there is an upper bound on the overall weights of cyclic subgroups of $\Gamma$, which completes the proof of Theorem \ref{thm:finiteMF}.

\section{Proof of Proposition \ref{prop:any-overall}}
\label{sec:any-overall-proof}

We now fill in the proof of Proposition \ref{prop:any-overall}.

Recall that the overall weight of a subgroup is
\[
\ow(I) = \sum_{(H,H')} \left( \frac{w(H,H')}{|I| \cdot |H'|} \sum_{\gamma \in \Gamma} \left( |\gamma I \gamma^{-1} \cap H| - |\gamma I \gamma^{-1} \cap H'| \right) \right).
\]
For notational convenience, let $w'(H,H') = \frac{w(H,H')}{|H'|}$, and $\ow'(I) = |I| \cdot \ow(I)$.
Also let
\[
u(I,(H,H')) = \sum_{\gamma \in \Gamma} \left| \gamma I \gamma^{-1} \cap (H \setminus H') \right|.
\]
Then equation (\ref{eqn:overall-defn}) simplifies to
\[
\ow'(I) = \sum_{(H,H')} w'(H,H') u(I,(H,H')).
\]

Let $[C_1],\ldots,[C_n]$ be the conjugacy classes of cyclic subgroups of $\Gamma$, and $P_1,\ldots,P_m$ be all the pairs $(H,H')$ such that $H'$ is maximal in $H$ and $H \subseteq \Gamma$.
Proposition \ref{prop:any-overall} is equivalent to the statement that given an ordered $n$-tuple of real numbers $(b_1,\ldots,b_n)$,
we can find $w'(P_j)$ for each $j$ such that $\sum_j w'(P_j) u(C_i,P_j) = b_i$ for each $i$.
In other words, the matrix $\left[ u(C_i,P_j) \right]$ (which has $u(C_i,P_j)$ in the $i$th row and $j$th column) has rank $n$.

Since we clearly have $m \ge n$, this is equivalent to the statement that the rows of this matrix are linearly independent.
That is, if there exist real numbers $a_{[C_i]}$, with
\begin{equation}
	\label{eqn:the-overall-system}
	\sum_{[C_i]} a_{[C_i]} u(C_i,P_j) = 0
\end{equation}
for all $j$, then $a_{[C_i]} = 0$ for all $i$.
This statement is what we will prove in the remainder of this section.

\subsection{Replacing Subgroups with Generating Elements}

For $C_i$ a cyclic subgroup of $\Gamma$, the number of conjugates of $C_i$ is $\frac{|\Gamma|}{|N_\Gamma(C_i)|}$, where $N_\Gamma(C_i)$ is the normalizer of $C_i$ in $\Gamma$.
Thus each conjugate of $C_i$ appears $|N_\Gamma (C_i)|$ times in the summation for $u(C_i,P_j)$.  Then the left-hand side of equation (\ref{eqn:the-overall-system}) becomes
\begin{equation}
	\label{eqn:all-cyclics}
	\sum_{[C_i]} a_{[C_i]} \sum_{\gamma} \left| \gamma C_i \gamma^{-1} \cap (H \setminus H') \right| = \sum_C a_{[C]} \cdot |N_\Gamma(C)| \cdot \left| C  \cap (H \setminus H') \right|,
\end{equation}
where the sum on the right side ranges over all cyclic subgroups of $\Gamma$.
The right side is also equal to
\[
\sum_C a_{[C]} \cdot |N_\Gamma(C)| \sum_{x \in C} \left| \{ x \}  \cap (H \setminus H') \right|,
\]
which, when we reverse the order of summation, becomes
\begin{equation}
	\label{eqn:rev-order-sum}
	\sum_{x \in \Gamma} \sum_{C \supseteq \langle x \rangle} a_{[C]} \cdot |N_\Gamma(C)| \cdot \left| \{ x \}  \cap (H \setminus H') \right|.
\end{equation}
Now if we set
\begin{equation}
	\label{eqn:a'-def}
	a'_{[\langle x \rangle]} = \sum_{C \supseteq \langle x \rangle} a_{[C]} \cdot |N_\Gamma(C)|
\end{equation}
for each $x \neq 1_\Gamma$, then equation (\ref{eqn:rev-order-sum}) can be written as
\[
\sum_{x \in \Gamma} a'_{[\langle x \rangle]} \left| \{ x \}  \cap (H \setminus H') \right| = \sum_{x \in H \setminus H'} a'_{[\langle x \rangle]}.
\]

We have now converted equation (\ref{eqn:the-overall-system}) into
\begin{equation}
	\label{eqn:the-overall-system-2}
	\sum_{x \in H \setminus H'} a'_{[\langle x \rangle]} = 0
\end{equation}
for all pairs $(H,H')$.

\subsection{Finishing the Proof}

To make use of equation (\ref{eqn:the-overall-system-2}), we need the following lemma:

\begin{lemma}
	If $a'_{[\langle x \rangle]} = 0$ for every $x \neq 1_\Gamma$, then $a_{[C]} = 0$ for every nontrivial cyclic $C \subseteq \Gamma$.
\end{lemma}

\begin{proof}
	We induct on the number of cyclic subgroups of $\Gamma$ containing $C$.
	If the only such subgroup is $C$ itself, then in equation (\ref{eqn:a'-def}), let $C = \langle x \rangle$.  We then obtain $a_{[C]} = \frac{1}{|N_\Gamma(C)|} a'_{[C]} = 0$.
	Otherwise,
	\[
	0 = a'_{[C]} = |N_\Gamma(C)| a_{[C]} + \sum_{C' \supsetneq C} a_{[C']} \cdot |N_\Gamma(C')|.
	\]
	Each $C'$ in the sum on the right is contained in fewer cyclic subgroups than $C$, so by the inductive hypothesis, $a_{[C']} = 0$, and the sum vanishes.
	It follows that $a_{[C]} = 0$, as desired.
\end{proof}

Now it suffices to prove that if 
\[
\sum_{x \in H \setminus H'} a'_{[\langle x \rangle]} = 0
\]
for each pair $(H,H')$, then $a'_{[\langle x \rangle]} = 0$ for every $x \neq 1_\Gamma$.
To do this, suppose that the order of $x \in \Gamma$ is $\prod p_k^{r_k}$ where the $p_k$ are distinct primes.
Let $S(x) = \sum r_k$.  We will induct on $S(x)$.

If $S(x) = 1$, then let $H = \langle x \rangle$ and $H' = 1$.
Since $H$ is cyclic of prime order, every nonidentity element of $H$ is a generator of $H$, so we have
\[
\sum_{x \in H \setminus H'} a'_{[\langle x \rangle]} = \left( |H| - 1 \right) a'_{[H]} = 0,
\]
and thus $a'_{[\langle x \rangle]} = a'_{[H]} = 0$, as desired.

Otherwise, let $H = \langle x \rangle$ and $H'$ be any maximal subgroup of $H$.
If $h \in \langle x \rangle$, then either $h$ generates $\langle x \rangle$, or the order of $h$ is a proper divisor of the order of $x$, from which $S(h) < S(x)$.
Thus either $a'_{[\langle h \rangle]} = a'_{[\langle x \rangle]}$ or $a'_{[\langle h \rangle]} = 0$ (by the inductive hypothesis).
Then
\[
\sum_{x \in H \setminus H'} a'_{[\langle x \rangle]} = B a'_{[\langle x \rangle]} = 0,
\]
where $B$ is the number of generators of $\langle x \rangle$.

Thus $a'_{[\langle x \rangle]} = 0$ for all $x \in \Gamma$.  This completes the proof of Proposition \ref{prop:any-overall}.

\section{Example: Mass Formulas for $D_4$ and $Q_8$}
\label{sec:mfd4}

To illustrate the implications of Theorem \ref{thm:finiteMF}, we now use it to find all weighted discriminant counting functions for $\Gamma = D_4$ that have universal mass formulas.

$D_4$ has four conjugacy classes of cyclic subgroups: two non-central copies of $C_2$, the center (also isomorphic to $C_2$), and one isomorphic to $C_4$.
We denote the overall weights of these by $w_{2a}$, $w_{2b}$, $w_{2c}$, and $w_4$, respectively.

\subsection{Overall Weights}

First, we must compute the overall weights of the three noncyclic subgroups of $D_4$ (itself and two nonconjugate copies of $V_4$)
in terms of $w_{2a}$, $w_{2b}$, $w_{2c}$, and $w_4$, using Proposition \ref{prop:overall-only}.

Let $I$ be the copy of $V_4$ generated by the subgroups whose overall weights are $w_{2a}$ and $w_{2c}$.
The $n=1$ term in equation (\ref{eqn:overall-cyclic}) contributes $\frac{1}{2} w_{2a} + \frac{1}{2} w_{2a} + \frac{1}{2} w_{2c}$ to $\ow(I)$,
since $I$ contains two conjugate copies of $C_2$ whose overall weight is each $w_{2a}$.
All other terms are zero, since the intersection of any two cyclic subgroups of $I$ is trivial.
Thus the overall weight of $I$ is
\[
\ow(I) = w_{2a} + \frac{1}{2} w_{2c}.
\]
By the same argument, the overall weight of the other copy of $V_4$ is $w_{2b} + \frac{1}{2} w_{2c}$.

Now let $I=D_4$.  The $n=1$ term in equation (\ref{eqn:overall-cyclic}) contributes $\frac{2}{4} w_{2a} + \frac{2}{4} w_{2b} + \frac{1}{4} w_{2c} + \frac{1}{2} w_4$.
The $n=2$ term contributes $-\frac{1}{4} w_{2c}$, and all other terms are zero.
Thus the overall weight of $I$ is
\[
\ow(D_4) = \frac{1}{2} w_{2a} + \frac{1}{2} w_{2b} + \frac{1}{2} w_4.
\]

\subsection{Extensions of $\Q_2$ and the Universality Equation}

Now we use \cite{Jones-DB} to find all maps from $G_{\Q_2}$ to $D_4$.
First, we list all field extensions of $\Q_2$ whose Galois group is a subgroup of $D_4$, and their (lower-numbered) ramification filtrations.  Each of these corresponds to several maps $G_{\Q_2} \to D_4$, and the number of maps per extension is the number of injections from the Galois group of the extension into $D_4$.
Note that several extensions have the same ramification filtration; since our counting functions do not depend on the Galois (i.e. decomposition) group, we can consider all such maps together.

The ramification filtrations listed below begin with the inertia group.  There are:
\begin{itemize}
	\item 2 $C_2$-extensions, 1 $C_4$-extension, and 1 $V_4$-extension with filtration $C_2 \supseteq C_2$, corresponding to 24 maps $G_{\Q_2} \to D_4$;
	\item 4 $C_2$-extensions, 2 $C_4$-extensions, and 2 $V_4$-extensions with filtration $C_2 \supseteq C_2 \supseteq C_2$, corresponding to 48 maps to $D_4$;
	\item 4 $V_4$-extensions with filtration $V_4 \supseteq V_4 \supseteq C_2 \supseteq C_2$, corresponding to 48 maps to $D_4$.
\end{itemize}

There are three conjugacy classes of subgroups of $D_4$ isomorphic to $C_2$; in the above list, the $C_2$'s in the ramification filtration map into each class equally often.  The same is true of the two conjugacy classes isomorphic to $V_4$.

There are also:
\begin{itemize}
	\item 8 $C_4$-extensions and 2 $D_4$ extensions with filtration $C_4 \supseteq C_4 \supseteq C_4 \supseteq C_2 \supseteq C_2$, corresponding to 32 maps to $D_4$;
	\item 2 $D_4$-extensions with filtration $V_4 \supseteq V_4$, corresponding to 16 maps to $D_4$;
	\item 2 $D_4$-extensions with filtration $V_4 \supseteq V_4 \supseteq C_2 \supseteq C_2$, corresponding to 16 maps to $D_4$;
	\item 8 $D_4$-extensions with filtration $D_4 \supseteq D_4 \supseteq V_4 \supseteq V_4 \supseteq C_2 \supseteq C_2$, corresponding to 64 maps to $D_4$;
	\item 4 $D_4$-extensions with filtration $D_4 \supseteq D_4 \supseteq C_4 \supseteq C_4 \supseteq C_2 \supseteq C_2 \supseteq C_2 \supseteq C_2$, corresponding to 32 maps to $D_4$.
\end{itemize}

In the second list, all $C_2$'s in the filtration must map into the center of $D_4$, but $V_4$ again maps equally often into each conjugacy class.

Theorem \ref{thm:mf-coeffs} allows us to calculate one side of the universality equation, and the above list gives the other side.
The universality equation is

\begin{align*}
	2^{-w_{2a}} &+ 2^{-w_{2b}} + 2^{-w_{2c}} + 2^{-w_4} \\
	& = 2^{-2w_{2a}} + 2^{-2w_{2b}} + 2^{-2w_{2c}} + 2 \cdot 2^{-3w_{2a}} + 2 \cdot 2^{-3w_{2b}} + 2 \cdot 2^{-3w_{2c}} \nonumber \\
	& + 2 \cdot 2^{-(3w_{2a} + w_{2c})} + 2 \cdot 2^{-(3w_{2b} + w_{2c})} + 2 \cdot 2^{-(2w_{2a} + 2w_{2c})}  \nonumber \\
	& + 2 \cdot 2^{-(2w_{2b} + 2w_{2c})} + 2^{-(2w_{2a} + w_{2c})} + 2^{-(2w_{2b} + w_{2c})} + 4 \cdot 2^{-(3w_4 + w_{2c})} \nonumber \\
	& + 4 \cdot 2^{-(w_{2a} + w_{2b} + 2w_4 + w_{2c})} + 4 \cdot 2^{-(2w_{2a} + w_{2b} + w_{2c} + w_4)} \nonumber \\
	&  + 4 \cdot 2^{-(w_{2a} + 2w_{2b} + w_{2c} + w_4)}. \nonumber
\end{align*}

\subsection{Results}

Following the proof of Theorem \ref{thm:finiteMF}, let $m$ be the greatest of $w_{2a}$, $w_{2b}$, $w_{2c}$, and $w_4$.
The 2-adic valuation of the left-hand side of the universality equation is at least $-m$.
The right-hand side is a sum of 16 powers of 2, and one of them has valuation at most $-3m$, so the right side cannot have valuation larger than $-3m + 15$.
Thus $-m \le -3m + 15$, from which $m \le 7$.

An exhaustive search (carried out using Sage) shows that the only positive integer solution to the universality equation is $w_{2a} = w_{2b} = 1$ and $w_{2c} = w_4 = 2$.
The corresponding counting function is the same as the one given by Wood in \cite{MMW-MF}, which comes from the wreath product structure of $D_4 \simeq C_2 \wr C_2$.

In fact, this counting function can also be constructed from the Artin conductor of the irreducible 2-dimensional representation of $D_4$, as described in Section \ref{sec:further-work}.  A recent result of Altug-Shankar-Varma-Wilson \cite{ASVW-D4} gives asymptotics for $D_4$ fields counted by the resulting alternate discriminant, which match the predictions of the Malle-Bhargava heuristics.

\subsection{Weighted Discriminants for $Q_8$}
\label{subsec:mfq8}
As a second example, let $\Gamma = Q_8$, the quaternion group.

As with $D_4$, $Q_8$ has four conjugacy classes of nontrivial cyclic subgroups.  Let $w_a$, $w_b$, and $w_c$ be the overall weights of the three cyclic subgroups of order 4, and $w_2$ be the overall weight of the subgroup of order 2 (which is the center of $Q_8$).  A calculation similar to that for $D_4$ gives the universality equation:

\begin{align*}
	2^{-w_a} &+ 2^{-w_b} + 2^{-w_c} + 2^{-w_2} \\
	&= 2^{-2w_2} + 2 \cdot 2^{-3w_2} + 4 \cdot 2^{-(w_2 + 3w_a)} + 4 \cdot 2^{-(w_2 + 3w_b)} + 4 \cdot 2^{-(w_2 + 3w_c)} \\
	&+ 4 \cdot 2^{-(2w_a + w_b + w_c)} + 4 \cdot 2^{-(w_a + 2w_b + w_c)} + 4 \cdot 2^{-(w_a + w_b + 2w_c)}.	
\end{align*}

If $m$ is the greatest overall weight, then the 2-adic valuation of the left-hand side of the universality equation is at least $-m$, and the valuation of the right-hand side is at most $-3m + 9$, and it follows that $m \le 4$.  Another exhaustive search shows that the only solution is $w_2 = w_a = w_b = w_c = 1$.  The corresponding counting function is also the Artin conductor of \textit{one-half} the irreducible 2-dimensional character of $Q_8$.

\begin{remark}
	There is no integer-valued weight function that produces these overall weights.  This illustrates why we require only the counting function to be integer-valued, rather than the weight function.
\end{remark}

\section{Further Questions}
\label{sec:further-work}

\subsection{Extending Theorem \ref{thm:finiteMF}}

It appears likely that Theorem \ref{thm:finiteMF} also holds for groups whose order is not a prime power.  In many small cases, including $C_6$, $S_3$, $C_{10}$, and $D_5$, the techniques of Section \ref{sec:big-proof} can be adapted \textit{ad hoc} to bound the overall weights. In fact, this may be possible whenever all the elements of $\Gamma$ have prime-power order.

However, groups like $C_{15}$ pose additional challenge.  Most importantly, there is no longer a single  ``universality equation''; instead, there is one such equation for each prime dividing $|\Gamma|$.  Additionally, Corollary \ref{cor:natural-double}, which was a key ingredient in bounding the exponent of each term in equation (\ref{eqn:MFF2}), fails when $\Gamma$ is not an $\ell$-group.

One could also ask if Theorem \ref{thm:finiteMF} holds over base fields other than $\Q$.  It appears likely that it does; transferring the definitions and supporting results to another base field should require very little modification except for Lemma \ref{lemma:extn-exist}.  In the definition of a mass formula, for instance, one would replace the fields $\Q_p$ by the nonarchimedean completions of the base field, and replace $p$ by the residue characteristic.

\subsection{Artin Conductors}

All the counting functions we have considered so far have been weighted discriminant counting functions, originating from a weight function.  However, there is another interesting class of alternate discriminants we could consider, originating from Artin conductors.  If $\chi$ is a character of $\Gamma$, then for any map $\phi : G_{\Q_p} \to \Gamma$, let
\[
 c_\chi (\phi) = \mathfrak{f} (\chi \circ \phi),
\]
where $\mathfrak{f}$ is the Artin conductor.  This $c_{\chi}$ is a proper counting function for $\Gamma$.

Based on Theorem \ref{thm:finiteMF}, and the rarity of universal mass formulas for small groups $\Gamma$ that we have examined in detail, we conjecture the following:
\begin{conj}
For any finite group $\Gamma$ (or at least any $p$-group), there are only finitely many characters $\chi$ for which the counting function $c_\chi$ has a universal mass formula.
\end{conj}

It would also be interesting to study the relationship between Artin-conductor counting functions and weighted discriminant counting functions.  For example, with $\Gamma = C_2 \times C_2$, every weighted discriminant counting function with integer overall weights is an Artin-conductor counting function, and vice versa (as long as we allow ``virtual characters'', linear combinations of irreducible characters with some negative coefficients).  We may even wish to allow non-integer coefficients, as in the case of $Q_8$ discussed above.

This is certainly not always the case, though; if $\Gamma$ does not have a rational character table, then $\Gamma$ has more characters than conjugacy classes of cyclic subgroups \cite{Serre:topics-in-GT}, so the lattice of Artin-conductor counting functions should have higher rank than the lattice of weighted discriminant counting functions.

%
%
%
%
%
%
%
%
%
%

\subsection{Infinite Weights}

Also of interest is expanding the definition of a weight function to allow $\infty$ as a weight.
Suppose $w(H,H') = \infty$.  It follows that for any $\phi:\Q_p \to \Gamma$, if 
\[
 \phi(I_p) \cap \gamma H' \gamma^{-1} \neq \phi(I_p) \cap \gamma H \gamma^{-1},
\]
where $I_p$ denotes the inertia subgroup of $G_{Q_p}$, then $\cw(\phi) = \infty$.
In terms of alternate discriminants, this means that any field $K$ in which any prime above $p$ is ramified in the extension $K_{H'}/K_H$ is assigned a value of $D_{\cw} = \infty$.  If we count number fields by $D_{\cw}$, then $K$ will be excluded from the count entirely.

Careful choices of which weights (or overall weights) to set equal to $\infty$ allows us to exclude fields with certain types of ramification.
Most notably, if $G$ and $A$ are finite groups with $A$ abelian, then we can in some cases use this technique to rephrase questions about unramified $A$-extensions of $G$-extensions of $\Q$
as questions about counting number fields by some alternate discriminant.

For example, if $G = C_2$ and $A$ is any finite abelian group, then any unramified $A$-extension of a quadratic field has Galois group $\Gamma = A \rtimes C_2$.
By choosing appropriate weights, we can turn the study of the $A$-moment of class groups of quadratic fields into the study of $\Gamma$-extensions of $\Q$, counted by an alternate discriminant.
Unfortunately, it is not always this simple to pin down what $\Gamma$ must be, so this technique becomes much harder to use for more complicated Cohen-Lenstra-type questions.


\begin{thebibliography}{9}
	
\bibitem{ASVW-D4}
S. A. Altu\u{g}, A. Shankar, I. Varma, and K. H. Wilson, ``The number of quartic $D_4$-fields ordered by conductor.''  Preprint: https://arxiv.org/abs/1704.01729.

\bibitem{B-W-sextic}
M. Bhargava and M. M. Wood, ``The density of discriminants of $S_3$-sextic number fields''.  \textit{Proc. Amer. Math. Soc.}, \textbf{136}:5 (2008), 1581-1587.

\bibitem{Dav-Heil}
H. Davenport and H. Heilbronn, ``On the density of discriminants of cubic fields II''.  \textit{Proc. Roy. Soc. London Ser. A}, \textbf{322}:1551 (1971), 405-420.

\bibitem{Jones-DB}
J. W. Jones and D. P. Roberts, \textit{Database of Local Fields} [online database].  Retrieved from http://math.la.asu.edu/~jj/localfields/.

\bibitem{Kedlaya-CF}
K. S. Kedlaya, ``Mass formulas for local Galois representations'', \textit{Int. Math. Res. Not.}, \textbf{2007}:17.

\bibitem{Krasner-NUM}
M. Krasner, ``Nombre des extensions d'un degr\'{e} donn\'{e} d'un corps $p$-adique'', in \textit{Les Tendances G\'{e}om\'{e}triques en Alg\`{e}bre et Th\`{e}orie des Nombres}. Paris, 1966.

\bibitem{malle:distro2}
G. Malle, ``On the distribution of Galois groups, II''.  \textit{Experiment. Math.}, \textbf{13}:2 (2004), 129-135.

\bibitem{Milne-KL}
J. S. Milne, \textit{Algebraic Number Theory} [lecture notes].  Retrieved from http://www.jmilne.org/math/CourseNotes/ANT210.pdf.

\bibitem{serre-LF}
J. P. Serre, \textit{Local Fields}.  (M. J. Greenberg, trans.).  New York: Springer-Verlag.

\bibitem{Serre:topics-in-GT}
J. P. Serre, \textit{Topics in Galois Theory}.  Notes by H. Darmon, http://www.math.mcgill.ca/~darmon/pub/Articles/Serre/c.pdf.

\bibitem{MMW-MF}
M. M. Wood, ``Mass formulas for local Galois representations to wreath products and cross products'', \textit{Algebra and Number Theory}, \textbf{2}:4 (2008), 391-405.

\end{thebibliography}
\end{document}